\documentclass[12pt]{article}

\usepackage{multicol}
\usepackage{graphicx}
\usepackage{amsmath}
\usepackage{amsthm}
\usepackage{amssymb}
\usepackage{amsfonts}
\usepackage{tikz-cd}
\usepackage[all]{xy}
\usepackage{enumerate}
\usepackage{multicol}
\usepackage[colorlinks=true,linkcolor=blue,citecolor=red]{hyperref}
\usepackage{lineno}

\newtheorem{theorem}{Theorem}[section]
\newtheorem{lemma}[theorem]{Lemma}
\newtheorem{proposition}[theorem]{Proposition}
\newtheorem{corollary}[theorem]{Corollary}

\newtheorem{remark}[theorem]{Remark}

\newcommand{\N}{\mathbb{N}}
\newcommand{\Z}{\mathbb{Z}}
\newcommand{\Q}{\mathbb{Q}}
\newcommand{\R}{\mathbb{R}}

\newcommand{\F}{\mathbb{F}}

\newcommand{\Ocal}{\mathcal{O}}
\newcommand{\Pcal}{\mathcal{P}}

\newcommand{\JR}{{\rm JR }}
\newcommand{\disc}{{\rm disc\: }}

\begin{document}

\title{Dedekind's criterion for the monogenicity of a number field versus Uchida's and L\"uneburg's}
\author{Xavier Vidaux and Carlos R. Videla}

\maketitle

\begin{abstract}
We compare three different characterizations, due respectively to R. Dedekind, K. Uchida, and H. L\"uneburg, of when $\Z[\theta]$ is the ring of integers of $\Q(\theta)$, and apply our results to some concrete $2$-towers of number fields.\footnote{The two authors have been partially supported by the first author Fondecyt project 1170315 from Conicyt, Chile.}
\end{abstract}





\section{Introduction}



A number field $K$ is said to be monogenic if it has a power integral basis, namely, if there exists $\theta$ in its ring of integers $\Ocal_K$ such that $\Ocal_K=\Z[\theta]$. Since Hensel's discovery in 1894 that the set of indices of integers of $K$ coincide with the set of absolute values of the integral values of a certain form over $\Z$ depending only on $K$ and a (given) integral basis, a lot of work has been done to study these so-called index forms --- see \cite[2.2, Section 6, p. 64]{Nark04} for details and a survey of results in this direction. In particular, by works of K\'alm\'an Gy\H{o}ry \cite{Gyory73}, it is decidable to know whether or not a number field is monogenic, as he gives an algorithm to list all possible generators of power basis, modulo some computable equivalence relation. Hence also it is decidable to know whether a given integral $\theta$ generates a power basis --- see \cite{EG17}. Nevertheless, for the latter problem, there are much more efficient algorithms, especially when one deals with multiple number fields simultaneously, or infinitely many. On a theoretical side, we found three different characterizations in the literature, due respectively to Dedekind, K\^oji Uchida, and Heinz L\"uneburg, of when $\Z[\theta]$ is the ring of integers of $\Q(\theta)$, and all the three are easily seen to be effective. In this paper, we will make explicit the connection between the three, which seems to be only partially known in the literature, and apply our results to some concrete $2$-towers of number fields, obtaining a local version of a recent theorem of Marianela Castillo \cite{Cas18} which links the question of monogenicity (for these towers) with a concrete problem in dynamical arithmetic. 

If $K=\Q(\theta)$ is a number field, $\Ocal_K$ its ring of integers, and $p$ a rational prime, it is well-known since Dedekind \cite{Dedekind1878} that the prime decomposition of $p\Ocal_K$ has the same shape as the prime decomposition of the minimal polynomial $f$ of $\theta$ modulo $p$, as long as $p$ does not divide the index of $\Z[\theta]$ in $\Ocal_K$. Much less known is the first general criterium for a prime $p$ not to divide the index, also involving the decomposition of $\bar f$, which appears in the same work (see for instance \cite{KumarKhanduja07} for an english version).  

\begin{theorem}[R. Dedekind, 1878]\label{ded1}
Let $\theta$ be an algebraic integer, $f$ its minimal polynomial, $K=\Q(\theta)$ and $\Ocal_K$ its ring of integers. Let $p$ be a rational prime. Let $\bar f=\varphi_1^{e_1}\dots\varphi_r^{e_r}$ be the decomposition of $\bar f$ in irreducible factors over the ring of polynomials $\F_p[t]$ over the field with $p$ elements. Let $\mu_i\in\Z[t]$ be any lifting of $\varphi_i$ and let $g\in\Z[t]$ be such that $f=\mu_1^{e_1}\dots\mu_r^{e_r}+pg$. The following are equivalent:
\begin{enumerate}
\item $\Z[\theta]$ is $p$-maximal (i.e. $p$ does not divide the index of $\Z[\theta]$ in $\Ocal_K$). 
\item For all $i$, either $e_i=1$ or $\varphi_i$ does not divide $\bar g$ in $\F_p[t]$.
\end{enumerate}
\end{theorem}


In 1977, K. Uchida \cite{Uch} proposed the following characterization and used it to give a beautiful short proof of the fact that the ring of integers of the cyclotomic field $\Q(\zeta_n)$ is $\Z[\zeta_n]$, where $\zeta_n$ is a primitive $n$-th root of unity.  

\begin{theorem}[K. Uchida, 1977]\label{uch1}
Let $R$ be a Dedekind ring. Let $\theta$ be an element of some integral domain which contains $R$ and assume that $\theta$ is integral over $R$. The ring $R[\theta]$ is a Dedekind ring if and only if, for any maximal ideal $M$ of the polynomial ring $R[t]$, the minimal polynomial $f(t)$ of $\theta$ over $R$ is not contained in $M^2$.
\end{theorem}

Note that Uchida does not require the finite norm property for $R$, whereas the generalizations of Dedekind's criterion seem to always require it (see \cite{Albu79} for a generalization of Uchida's result). 

Finally, in 1984, H. L\"uneburg \cite{Lu84}, unaware of Uchida's work, discovered another test based on the well known fact that a ring is a Dedekind domain if and only if it is noetherian of Krull dimension $1$ and its localization at every prime is a discrete valuation ring \cite[Theorem 9.3, p. 95]{AtMac69}. He also uses his test to give an alternative proof of the fact that $\Z[\zeta_n]$ is the ring of integers of $\Q(\zeta_n)$, which is more involved than that of Uchida, although it is actually his general theorem which will be useful for studying our $2$-towers. 

If $\Pcal$ is a maximal ideal of $\Z[\theta]$, with $\theta$ integral, we denote by $\mu_\Pcal$ the monic polynomial of least degree such that $\mu_\Pcal(\theta)\in\Pcal$. 

\begin{theorem}[H. L\"uneburg, 1984]\label{Hilfsaatzif}
Let $\theta$ be an algebraic integer and $f$ its minimal polynomial. Let $\Pcal$ be a maximal ideal of $\Z[\theta]$. Let $p$ be the rational prime below $\Pcal$. If there exists $g,h\in\Z[t]$ such that $f=\mu_\Pcal h+pg$ and $\gcd(\bar \mu_\Pcal,\bar g,\bar h)=1$ in $\F_p[t]$, then the localization $\Z[\theta]_\Pcal$ of $\Z[\theta]$ at $\Pcal$ is a discrete valuation ring.
\end{theorem}

Note the similarities between L\"uneburg's test and Dedekind's criterion. 

Using L\"uneburg's test, and also unaware of Uchida's work, K. Yamagata and M. Yamagishi \cite{YY} gave an alternative proof of the fact that the real subfield of any finite cyclotomic extension of $\Q$ is monogenic, with ring of integers $\Z[\zeta_n+\zeta_n^{-1}]$ --- a result that was apparently first proved by J. J. Liang \cite{Liang76}, and for which a very short proof is known \cite[Prop. 2.16, p. 16]{Washington97}. 

The relationship between the three approaches begs for an explanation. This is the main objective of this work. In Section \ref{UGL}, we will uncover the relation between the witness $M$ for not being a Dedekind domain in Uchida's approach, the witness $\Pcal$ for not being a discrete valuation ring in L\"uneburg's test, and the witnesses $p$ and $i$ for not being the ring of integers in Dedekind's approach. 

We will be able to do that thanks to the following theorem, which will be proven in Section \ref{secHilfsaatz} (with the notation as in Theorem \ref{Hilfsaatzif}, we will prove in Lemma \ref{lemgenmin} that $\bar\mu_\Pcal$ is irreducible and divides $\bar f$). 

\begin{theorem}\label{Hilfsaatziffintro}
Let $\theta$ be an algebraic integer and $f$ its minimal polynomial. Let $\Pcal$ be a maximal ideal of $\Z[\theta]$. Let $p$ be the rational prime below $\Pcal$. The following statements are equivalent: 
\begin{enumerate}
\item The localization $\Z[\theta]_{\Pcal}$ is a discrete valuation ring.
\item For every $g,h\in\Z[t]$ such that $f=\mu_\Pcal h+pg$, we have $\gcd(\bar\mu_\Pcal,\bar g,\bar h)=1$.
\item There exists $g,h\in\Z[t]$ such that $f=\mu_\Pcal h+pg$ and $\gcd(\bar \mu_\Pcal,\bar g,\bar h)=1$.
\item $f\notin (p,\mu_\Pcal(t))^2$.
\end{enumerate}
\end{theorem}

Note that since $\bar \mu_\Pcal$ is irreducible, the condition $\gcd(\bar \mu_\Pcal,\bar g,\bar h)=1$ is equivalent to the condition: $\gcd(\bar \mu_\Pcal,\bar h)=1$ or $\gcd(\bar \mu_\Pcal,\bar g)=1$.

We will prove first that 3 implies 1 (for sake of completeness, as it is L\"uneburg's theorem), then that 4 implies 2, then that 1 is equivalent to 4, and finally that 1 implies 3. Note that item 2 implies trivially item 3 because $\bar\mu_\Pcal$ divides $\bar f$. 

Theorema \ref{Hilfsaatziffintro} can be seen as a local version of Uchida's theorem through L\"uneburg's criterium. As we were finalizing the writing of this note, we discovered the following result \cite[Lemma 2.1]{JKS16}, which can also be seen as a local version of Uchida's theorem, but through Dedekind's criterium:
 
\begin{theorem}[Jakhar-Khanduja-Sangwan, 2016]
Under the same notation as in Dedekind's criterium: $\Z[\theta]$ is $p$-maximal if and only if for any $i$ we have $f\notin(p,\mu_i(t))^2$. 
\end{theorem}

After we connect all the witnesses in Section \ref{UGL}, the relation between the above result and ours also will become clear.

In Section \ref{torre}, we will apply Theorem \ref{Hilfsaatziffintro} to the following towers of nested square roots. Given $\nu\ge2$ an integer and $x_0=0$, write $x_{n+1}=\sqrt{x_n+\nu}$. Write $\Z^{(\nu)}=\cup_n\Z[x_n]$, and $K^{\nu}=\cup_n\Q(x_n)$. The initial motivation for considering these towers was a problem in Logic going back to Julia Robinson \cite{Rob59}. Given a totally real subfield $K$ of an algebraic closure $\tilde\Q$ of $\Q$ and $t\in\R\cup\{+\infty\}$, she defines 
$$
\Ocal_K^{<<t}=\{\theta\in \Ocal_K\colon 0<<\theta<<t\}
$$
and the $\JR$ number 
$$
\JR(\Ocal_K)=\inf\{t\in\R\cup\{+\infty\}\colon \Ocal_K^{<<t}\textrm{ is infinite}\},
$$ 
where $0<<\theta<<t$ stands for ``$\theta$ and all its conjugates are in the real interval from $0$ to $t$'', and she proves that if $\JR(\Ocal_K)$ is a minimum or $+\infty$, then the semi-ring of natural numbers is first-order definable in $\Ocal_K$ (hence its theory is undecidable). Julia Robinson asked whether the $\JR$ number is always a minimum. In \cite{VyV}, we compute explicitly the $\JR$ number of $\Z^{(\nu)}$ and prove that it is a minimum for infinitely many values of $\nu$, whereas it is not a minimum for infinitely many values of $\nu$, though in the latter case it has another topological property that ensures the definability of $\N$. In all cases the $\JR$ number lies strictly between $4$ and $+\infty$. We asked there for examples of $\Ocal_K$ with $\JR$ number strictly between $4$ and $+\infty$, which led us to the question whether any of these $\Z^{(\nu)}$ is integrally closed. 
P. Gillibert and G. Ranieri \cite{GiRan17} constructed infinitely many $\Ocal_K$ for which the $\JR$ number is strictly between $4$ and $+\infty$. In all cases it turns out to be a minimum, leaving wide open the question of Julia Robinson. In the mean time, with M. Castillo \cite{CVV18}, we constructed infinitely many $\Ocal_K$ for which either the $\JR$ is $4$ and it is not a minimum, or it is neither $4$ nor $+\infty$.

Let $P_n$ be the minimal polynomial of $x_n$ and $C_n$ its constant term. In her thesis \cite{Cas18}, M. Castillo proves, using Uchida's theorem, that for infinitely many $\nu$, $\Z^{(\nu)}$ is integrally closed if and only if $C_n$ is square-free for every $n$. With the same notation as above, we obtain the following local version of her theorem. 

\begin{theorem}\label{p2intro}
Let $p$ be a prime number which divides some $C_n$, and $n(p)$ be the smallest such index $n$. 
\begin{enumerate}
\item If $p^2$ does not divide $C_{n(p)}$, then for any $\Pcal$ maximal over $p\Z$ and for any $n$, the localization of $\Z[x_n]$ at $\Pcal$ is a discrete valuation ring. 
\item If for any $\Pcal$ maximal over $p\Z$ the localization of $\Z[x_{n(p)}]$ at $\Pcal$ is a discrete valuation ring then $p^2$ does not divide $C_{n(p)}$.
\end{enumerate}
\end{theorem}

Note that item 1 says in particular that the information at the level $n(p)$ is enough to know what happens at every level of the tower with respect to the chosen prime $p$. 

Before we can apply Theorem \ref{Hilfsaatziffintro} to prove Theorem \ref{p2intro}, we need first a few results on the dynamic of the sequences of $C_n$ and $P_n$ --- for instance, we will see that the sequence $(C_n)$ is a super-rigid divisibility sequence, meaning that if a prime $p$ divides some $C_n$ with order $v$, then it divides $C_{n+kp}$ for every $k$ with the same order.



\section{A local version of Uchida's Theorem and the reciprocal to L\"uneburg's Test}\label{secHilfsaatz}

Everything in this section up to Remark \ref{UchLunRem} is essentially due to L\"uneburg \cite{Lu84}. We include it for the sake of completeness, in order to fix ideas and notation, and because we were not able to find a reference in the english language. 

In all this section: 
\begin{itemize}
\item $\theta$ is an algebraic integer
\item $R=\Z[\theta]$
\item $f(x)$ is the minimal polynomial of $\theta$ over $\Z$
\item $n$ is the degree of $f$
\item $\Pcal$ is a maximal ideal of $R$
\item $p$ is the prime in $\Z$ below $\Pcal$, i.e. such that $p\Z=\Pcal\cap\Z$
\item $L=R_\Pcal$ is the localization of $R$ at $\Pcal$
\item $\mu=\mu_\Pcal$ is a monic polynomial over $\Z$ of least degree such that $\mu(\theta)\in\Pcal$. 
\item $I=pR+\mu(\theta)R$ (we will prove that $I=\Pcal$)
\item $M=p\Z[t]+\mu\Z[t]$
\item $Q=pL+\mu(\theta)L$. 
\end{itemize}

Write $\pi_1\colon \Z\rightarrow\F_p$ and $\pi_2\colon\Z[\theta]\rightarrow\Z[\theta]/\Pcal$ for the canonical projections and $i_1\colon\Z\rightarrow\Z[\theta]$ for the inclusion map, which induces $i_2\colon\F_p\rightarrow\Z[\theta]/\Pcal$ such that $\pi_2\circ i_1=i_2\circ\pi_1$: 
$$
\begin{tikzcd}
\Z[\theta] \arrow[r, twoheadrightarrow, "\pi_2"] 
& \Z[\theta]/\Pcal \\
\Z \arrow[u, hook, "i_1"] \arrow[r, twoheadrightarrow, "\pi_1"]
&\F_p \arrow[u, hook, "i_2"']
\end{tikzcd}
$$
For $m\in\Z$, we may write $\pi_1(m)=m+p\Z=\bar m$, and for $\alpha\in\Z[\theta]$ we may write $\pi_2(\alpha)=\alpha+\Pcal=\tilde \alpha$. We may evaluate polynomials over $\Z$ at elements of $\Z[\theta]$, and polynomials over $\F_p$ at elements of $\Z[\theta]/\Pcal$. Also, for $\mu\in\Z[t]$, we may write $\bar\mu\in\F_p[t]$ for the reduced polynomial modulo $p$. 

\begin{lemma}\label{lemgenmin}
The minimal polynomial of $\pi_2(\theta)$ is $\bar\mu$. 
\end{lemma}
\begin{proof}
Let $\varphi_0\in\F_p[t]$ be the minimal polynomial of $\pi_2(\theta)$. Since $\mu(\theta)\in\Pcal$, we have $\bar\mu(\pi_2(\theta))=\pi_2(\mu(\theta))=0+\Pcal$, so $\varphi_0$ divides $\bar\mu$ in $\F_p[t]$. Let $\mu_0\in\Z[t]$ be a lift of $\varphi_0$. We have 
$$
\deg(\mu_0)=\deg(\varphi_0)\le\deg(\bar\mu)=\deg(\mu)
$$ 
because $\varphi_0$ and $\mu$ are monic by hypothesis. Since $\pi_2(\mu_0(\theta))=\varphi_0(\pi_2(\theta))=0+\Pcal$, we have $\mu_0(\theta)\in\Pcal$. Since $\mu_0$ is monic, we deduce $\deg(\mu)\le\deg(\mu_0)$ (from the minimality property for $\mu$). Therefore, we have 
$$
\deg(\bar\mu)=\deg(\mu)\le\deg(\mu_0)=\deg(\varphi_0),
$$ 
hence $\deg(\varphi_0)=\deg(\bar\mu)$, so $\varphi_0=\bar\mu$. 
\end{proof}

We know by Lemma \ref{lemgenmin} that $\bar\mu$ is one of the irreducible factors of $\bar f$, but which factor it is will depend in general on the choice of $\Pcal$ above $p$. 

\begin{lemma}\label{lemmuirr}
The polynomial $\mu$ is irreducible in $\Q[t]$. 
\end{lemma}
\begin{proof}
If $\mu=rs$, with $r,s\in\Z[t]$, then $r(\theta)s(\theta)\in\Pcal$, which is a prime ideal of $R$. Without loss of generality, assume that $r(\theta)\in\Pcal$. Because of the minimality condition on $\mu$, the degree of $r$ is at least the degree of $\mu$, hence $r$ and $\mu$ have the same degree and $s=\pm 1$. We conclude by Gauss' Lemma. 
\end{proof}

Next lemma will be used in this section only with $F=\mu$, but we present it in this slightly more general form as we will need it in the next section. 

\begin{lemma}\label{Hilfsaatz0a}
For any $F\in\Z[t]$ such that $\bar F$ is irreducible and $\bar F$ divides $\bar f$, the ideal $pR+F(\theta)R$ is prime. 
\end{lemma}
\begin{proof}
Since $\bar F$ divides $\bar f$, there exist $h,g\in\Z[t]$ such that $f=F h+pg$. 

Write $J=pR+F(\theta)R$. We use Dedekind's argument. Let $a(t)=\sum_{i=0}^{n-1}a_it^i$ and $b(t)=\sum_{i=0}^{n-1}b_it^i$ be polynomials in $\Z[t]$ such that $a(\theta)b(\theta)\in J$. So there exist $W_1,W_2\in\Z[t]$ of degree at most $n-1$ such that 
$$
pW_1(\theta)+F(\theta)W_2(\theta)=a(\theta)b(\theta).
$$ 
Since $ab-pW_1-F W_2$ evaluated in $\theta$ is $0$ and $f$ is the minimal polynomial of $\theta$, there exists $W_3\in\Z[t]$ such that $ab=pW_1+F W_2+fW_3$. Since $f=F h+pg$, we can write $ab=pW_1^*+F W_2^*$ for some $W_1^*,W_2^*\in\Z[t]$. So $\bar F$ divides $\bar a\bar b$ in $\F_p[t]$, and since it is irreducible, it divides either $\bar a$ or $\bar b$. Without loss of generality, suppose that it divides $\bar a$. We have then $a(t)=F(t)W_4(t)+pW_5$ for some $W_4,W_5\in\Z[t]$, so $a(\theta)\in pR+F(\theta)R=J$. 
\end{proof}

\begin{lemma}[\cite{Lu84}, proof of Hilfssatz 4]\label{Hilfsaatz0}
We have $pR+\mu(\theta)R=\Pcal$. 
\end{lemma}
\begin{proof}
Clearly we have $I=pR+\mu(\theta)R\subseteq\Pcal$, so it is enough to show that $I$ is prime, because it is known that prime ideals in an integral extension of $\Z$ are maximal (see for instance \cite[Cor. 5.8, p. 61]{AtMac69}). We conclude by Lemma \ref{Hilfsaatz0a}, because we know from Lemma \ref{lemmuirr} that $\bar\mu$ is irreducible, and from Lemma \ref{lemgenmin} that it divides $\bar f$. 
\end{proof}

We can now prove Theorem \ref{Hilfsaatzif}, that we state now in the context of this section. 

\begin{theorem}[\cite{Lu84}, Hilfssatz 4]\label{Hilfsaatz}
If there exists $h,g\in\Z[t]$ such that $f=\mu h+pg$ and $\gcd(\bar \mu,\bar g,\bar h)=1$ in $\F_p[t]$ then $L$ is a discrete valuation ring.
\end{theorem}
\begin{proof}
We know by Lemma \ref{Hilfsaatz0} that $I=\Pcal$. Write $Q=pL+\mu(\theta)L$ for the (unique) maximal ideal of $L$. 

We prove that $Q$ is principal. Since $\bar\mu$ is irreducible, the condition $\gcd(\bar \mu,\bar g,\bar h)=1$ is equivalent to having either $\gcd(\bar \mu,\bar g)=1$ or $\gcd(\bar \mu,\bar h)=1$. 

Assume first that $\gcd(\bar \mu,\bar h)=1$, so that $h(\theta)\notin\Pcal$. Let $y\in Q$ and write $y=(p\alpha+\mu(\theta)\beta)/\gamma$,  where $\alpha,\beta\in R$ and $\gamma\in R\setminus\Pcal$. Since $f(\theta)=0$, we have $\mu(\theta)h(\theta)=-pg(\theta)$, hence 
$$
y=p\frac{\alpha h(\theta)-g(\theta)\beta}{\gamma h(\theta)},
$$
and since $\Pcal$ is prime, $\gamma h(\theta)$ does not lie in $\Pcal$, so $y\in pL$. So in that case we have $Q=pL$. 

Assume now that $\gcd(\bar \mu,\bar g)=1$, so that $g(\theta)\notin\Pcal$. Hence 
$$
p=\frac{-\mu(\theta)h(\theta)}{g(\theta)}\in\mu(\theta)L,  
$$
so that $Q=\mu(\theta)L$. 

Since $R$ is a finite integral extension of $\Z$, it is a noetherian domain, and all its non-zero prime ideals are maximal --- by \cite[Cor. 5.8, p. 61]{AtMac69}. In particular, $R$ has Krull dimension $1$, hence also $L$. Since $L$ is noetherian, has Krull dimension $1$, and its maximal ideal is principal, we conclude that it is a discrete valuation ring --- see \cite[Prop. 9.2, p. 94]{AtMac69}. 
\end{proof}

\begin{remark}\label{UchLunRem}
There are some situations that trivially imply that the hypothesis of Theorem \ref{Hilfsaatz} is satisfied: 
\begin{enumerate}
\item When $p$ does not divide the discriminant of $\theta$: this is because in that case the reduced polynomial $\bar f$ is separable, so for whatever decomposition $f=\mu h+pg$, the reduced equation $\bar f=\bar\mu\bar h$ will give $\gcd(\bar\mu,\bar h)=1$. 
\item If $\bar\mu$ lies in the square-free part of $\bar f$, then trivially $\gcd(\bar\mu,\bar h)=1$. 
\item If $\bar g$ is a non-zero constant.  
\end{enumerate}
\end{remark}

\begin{lemma}\label{UchLunLem}
If $\bar f$ has multiplicity $\ge 2$ at $\bar\mu$ and $L$ is a discrete valuation ring, then $f\notin M^2$ and there exists $h,g\in\Z[t]$ such that $f=\mu h+pg$ and $\gcd(\bar\mu,\bar g)=1$. 
\end{lemma}
\begin{proof}
Since $L$ is a discrete valuation ring, in particular $Q=yL$ is principal and we have 
$$
pL=y^k L, \quad\textrm{and}\quad \mu(\theta)L=y^\ell L
$$ 
for some integers $k,\ell\ge1$. 

If both $k$ and $\ell$ are greater than $1$, then for any $u,v\in L$, the order of $pu+\mu(\theta)v$ at $Q$ would be at least $2$, which is impossible. We now show that $k$ cannot be $1$. Assume it is, so that there exist polynomials $a,b\in \Z[t]$ such that $b(\theta)\in R\setminus \Pcal$ and 
$$
\mu(\theta)=\frac{a(\theta)}{b(\theta)}p^\ell.
$$
Since $f$ is the minimal polynomial of $\theta$, there exists $W\in\Z[t]$ such that $p^\ell a-\mu b=fW$, hence $-\bar\mu \bar b=\bar f \bar W=\bar\mu^2\bar hW$ for some $h$ (which exists by hypothesis), and we have $\bar b=-\bar\mu\bar h\bar W$, so $b(\theta)$ lies in $\Pcal$, a contradiction. 

We deduce that $\ell=1$, hence $Q=\mu(\theta)L$. We have 
$$
p=\frac{a(\theta)}{b(\theta)}\mu(\theta)^k
$$
for some polynomials $a,b\in\Z[t]$ such that $b(\theta)\notin\Pcal$. The polynomial $\mu^k a-pb\in\Z[t]$ takes the value $0$ at $\theta$, hence it is divisible by $f$. Write 
$$
fW_1=\mu^k a-pb,
$$
where $W_1\in\Z[t]$. On the other hand, since $\bar f$ has multiplicity $\ge 2$ at $\bar\mu$, there exists $W_2\in\Z[t]$ such that  
$$
f=\mu h+pg=\mu(\mu W_2)+pg. 
$$
We now prove that $\gcd(\bar\mu,\bar g)$ is $1$. We have
$$
fW_1=\mu^2W_2W_1+pgW_1=\mu^k a-pb\in\Z[t],
$$
hence 
$$
\mu^2(W_2W_1-\mu^{k-2}a)=p(gW_1-b),
$$ 
so $\mu$ divides $gW_1-b$, hence $gW_1-b=\mu W_3$ for some $W_3\in\Z[t]$. If $\bar\mu$ would divide $\bar g$, then we would have $g=\mu W_4+pW_5$ for some $W_4,W_5\in\Z[t]$, hence $(\mu W_4+pW_5)W_1-b=\mu W_3$, hence 
$$
pW_5(\theta)W_1(\theta)-b(\theta)\in\Pcal,
$$
and since $pW_5(\theta)W_1(\theta)\in\Pcal$, this would contradict the definition of $b(\theta)$. 

For the sake of contradiction, suppose that $f$ lies in $M^2$, so we can write 
$$
f=\mu^2 W_6+p\mu W_7+p^2W_8=\mu^2W_6+p(\mu W_7+pW_8)
$$ 
for some polynomials $W_6,W_7,W_8\in\Z[t]$. Therefore, we have $\bar f=\bar\mu^2\bar W_2=\bar\mu^2\bar W_6$, hence $W_6=W_2+pW_9$ for some $W_9\in\Z[t]$, and 
$$
f=\mu^2W_2+p(\mu(W_7-\mu W_9)+pW_8), 
$$
hence $\bar g=\bar\mu(\bar W_7-\bar\mu \bar W_9)$, which contradicts the fact that $\bar\mu$ and $\bar g$ are coprime. 
\end{proof}

\begin{lemma}\label{UchLunLem2}
If $f\notin M^2$, then for every $g,h\in\Z[t]$ such that $f=\mu h+pg$ we have $\gcd(\bar\mu,\bar h,\bar g)=1$. 
\end{lemma}
\begin{proof}
Let $g,h$ be such that $f=\mu h+pg$ and $\gcd(\bar\mu,\bar g,\bar h)\ne1$. We have then $\gcd(\bar\mu,\bar g,\bar h)=\bar\mu$, so we can write 
$$
f=\mu(\mu h_0+pu)+p(\mu g_0+pv)=\mu^2h_0+p\mu(u+g_0)+p^2v\in M^2.
$$
\end{proof}

The following proposition proves that item 1 is equivalent to item 4 in Theorem \ref{Hilfsaatziffintro}. 

\begin{proposition}\label{UchLun}
The ring $L$ is a discrete valuation ring if and only if $f\notin M^2$. 
\end{proposition}
\begin{proof}
If $f\notin M^2$, then by Lemma \ref{UchLunLem2}, there exists $g,h$ such that $f=\mu h+pg$ and $\gcd(\bar\mu,\bar g,\bar h)=1$. We can conclude by Theorem \ref{Hilfsaatz} that $L$ is a discrete valuation ring. 

Assume that $L$ is a discrete valuation ring. If $\bar\mu^2$ divides $\bar f$, then we are done by Lemma \ref{UchLunLem}. Suppose that $\bar\mu^2$ does not divide $\bar f$. If $f$ would be in $M^2$, then we would have 
$$
f=\mu^2A+p\mu B+p^2 C
$$
for some $A,B,C\in\Z[t]$, so reducing modulo $p$, we would have $\bar f=\bar\mu^2 \bar A$, which is a contradiction (note that $\bar A$ cannot be $0$ because $f$ is monic). 
\end{proof}

The following corollary finishes the proof of Theorem \ref{Hilfsaatziffintro}. 

\begin{corollary}\label{Hilfsaatziff}
If $L$ is a discrete valuation ring, then there exist $g,h\in\Z[t]$ such that $f=\mu h+pg$ and $\gcd(\bar \mu,\bar g,\bar h)=1$ in $\F_p[t]$.
\end{corollary}
\begin{proof}
If $\gcd(\bar \mu,\bar h)=1$, then there is nothing to prove. If not, then $\bar\mu$ divides $\bar h$, so $\bar f$ has multiplicity $\ge 2$ at $\bar\mu$. Therefore, $\gcd(\bar\mu,\bar g)=1$ by  Lemma \ref{UchLunLem}.
\end{proof}


\section{Uchida versus L\"uneburg versus Dedekind}\label{UGL}

In all this section, we will assume that $\theta$ is an algebraic integer with minimal polynomial $f$, that $\Z[\theta]$ is not integrally closed, so we have a witness $\Pcal$ of L\"uneburg, a witness $M$ of Uchida, and a witness $(p,i)$ of Dedekind. We summarize the connections of the three criteria as follows. 
\begin{enumerate}
\item Given a witness $(p,i)$ for Dedekind, for any lift $\mu_i$ of $\varphi_i$ the ideals $(p,\mu_i(t))$ and $(p,\mu_i(\theta))$ are witnesses respectively of Uchida and L\"uneburg. 
\item Given a witness $M=(p,F(t))$ for Uchida, the ideal $\Pcal=(p,F(\theta))$ is a witness for L\"uneburg (and $\mu_\Pcal=F$), and for $i$ such that $\bar F=\varphi_i$, the pair $(p,i)$ is a witness for Dedekind. 
\item Given a witness $\Pcal$ for L\"uneburg, for $p$ below $\Pcal$, the ideal $(p,\mu_\Pcal(t))$ is a witness for Uchida, and for $i$ such that $\bar\mu_\Pcal=\varphi_i$, the pair $(p,i)$ is a witness for Dedekind. 
\end{enumerate}



\noindent\emph{From $\Pcal$ to $(p,i)$.} 

Suppose that you know $\Pcal$ such that for all $g,h\in\Z[t]$ such that $f=\mu_\Pcal h+pg$, we have $\gcd(\bar \mu_\Pcal,\bar g,\bar h)\ne1$ in $\F_p[t]$. Let $p$ be the rational prime below $\Pcal$ and $\bar f=\prod\varphi_i^{e_i}$ be the decomposition of $\bar f$ into irreducibles. Since $\bar\mu_\Pcal$ is irreducible and divides $\bar f$, we have $\bar\mu_\Pcal=\varphi_{i_0}$ for some $i_0$. Choose $\mu_{i_0}=\mu_\Pcal$, and the other $\mu_i$'s and $g$ arbitrary so that $f=\prod\mu_i^{e_i}+pg$. With $h=\mu_{i_0}^{e_{i_0-1}}\prod_{i\ne i_0}\mu_i^{e_i}$, we have $f=\mu_\Pcal h+pg$, so $\gcd(\bar \mu_\Pcal,\bar g,\bar h)\ne1$. Since $\bar\mu_\Pcal=\varphi_{i_0}$ is irreducible, it divides both $\bar g$ and $\bar h$, and in particular we have $e_{i_0}\ge2$. So $(p,i_0)$ is a Dedekind's witness. \\

\noindent\emph{From $(p,i_0)$ to $\Pcal$.} 

Suppose now that you know a Dedekind's witness $(p,i_0)$. So, writing $\bar f=\prod\varphi_i^{e_i}$, for any lift $\mu_i$ of each prime divisor $\varphi_i$ of $\bar f$ in $\F_p[t]$, and for any $g$ such that $f=\prod\mu_i^{e_i}+pg$, we have $e_{i_0}\ge2$ and $\varphi_{i_0}\mid \bar g$. Since $\varphi_{i_0}$ is irreducible and divides $\bar f$, the ideal $\Pcal_{i_0}=(p,\mu_{i_0}(\theta))$ is a prime ideal of $\Z[\theta]$ (above $p$) by Lemma \ref{Hilfsaatz0a}, hence it is maximal, for whatever lift $\mu_{i_0}$ we choose. Let $\mu$ be the monic polynomial of least degree such that $\mu(\theta)\in\Pcal_{i_0}$ and let $\Pcal=(p,\mu(\theta))$. Since $\bar\mu$ is irreducible and divides $\bar f$, $\Pcal$ is maximal by Lemma \ref{Hilfsaatz0a}. Since $\Pcal\subseteq\Pcal_{i_0}$, we have $\Pcal=\Pcal_{i_0}$. Moreover, since $\bar\mu$ is irreducible, it is equal to some $\varphi_j$. By the same argument as above, we have $\Pcal=\Pcal_j=(p,\mu_j(\theta))$ for any lift $\mu_j$ of $\varphi_j$. In particular, we have $\Pcal_j=\Pcal_{i_0}$. An argument of Dedekind allows to conclude that $j=i_0$. Suppose not. Since $\varphi_{i_0}$ and $\varphi_j$ are coprime, there exist $a,b\in\F_p[t]$ such that $a\varphi_{i_0}+b\varphi_j=1$. So for any lift, there exist $A,B\in\Z[t]$ and $C\in p\Z[t]$ such that $A\mu_{i_0}+B\mu_j=1+C$, so in particular $A(\theta)\mu_{i_0}(\theta)+B(\theta)\mu_j(\theta)=1+C(\theta)$, hence $1\in\Pcal$, which is absurd. We deduce that $\bar\mu=\varphi_j=\varphi_{i_0}$. So we can choose $\mu_{i_0}=\mu$. Choose any other lift for the rest of the $\mu_i$ and for $g$ so that $f=\prod\mu_i^{e_i}+pg$. For $h=\mu_{i_0}^{e_{i_0-1}}\prod_{i\ne i_0}\mu_i^{e_i}$ we have $f=\mu h+pg$, and the conclusion of Dedekind  $e_{i_0}\ge2$ and $\varphi_{i_0}\mid \bar g$ gives $\bar\mu\mid\bar h$ and $\bar\mu\mid \bar g$, hence $\gcd(\bar\mu,\bar h,\bar g)\ne 1$. \\

\noindent\emph{From $\Pcal$ to $M$.} 

By Theorem \ref{Hilfsaatziffintro} (4 implies 1), Uchida's witness $M$ is just the ideal $(p,\mu_\Pcal(t))$ in $\Z[t]$. \\

\noindent\emph{From $M$ to $\Pcal$.} 

Let $M$ be a maximal ideal of $\Z[t]$ such that $f\in M^2$, where $f$ is the minimal polynomial of $\theta$. It is well known that $M$ has to be of the form $(p,F(t))$, where $p$ is a rational prime and $F$ is a monic polynomial in $\Z[t]$ whose reduction $\bar F$ modulo $p$ is irreducible in $\F_p[t]$. Consider the factorization
$$
\bar f=\prod_{i=1}^n \varphi_i^{e_i}
$$
of $\bar f$ as a product of irreducible polynomials in $\F_p[t]$. Since $f\in M^2$, we have $\bar f=\psi \bar F^2$ for some $\psi\in\F_p[t]$. Note that $\psi$ is not the zero polynomial because $\bar f$ is monic. Therefore, we have
$$
\prod_{i=1}^n \varphi_i^{e_i}=\psi \bar F^2, 
$$
and we deduce that there exists a unique $i$ such that $\bar F=\varphi_i$. Let $\Pcal$ be the ideal $(p,F(\theta))$ of $\Z[\theta]$, which we know to be prime, hence maximal, by Lemma \ref{Hilfsaatz0a}. We prove now that $F$ is indeed $\mu_\Pcal$. By Lemma \ref{Hilfsaatz0}, we have $(p,\mu_\Pcal(\theta))=\Pcal$, hence $(p,\mu_\Pcal(\theta))=(p,F(\theta))$, so in particular we have $\mu_\Pcal(\theta)=\alpha p+\beta F(\theta)$ for some polynomials $\alpha, \beta\in\Z[t]$. In particular, we have $\bar\mu_\Pcal=\bar\beta\bar F$, and since $\bar\mu_\Pcal$ is irreducible, we deduce that $\bar\mu_\Pcal=\bar F$. We conclude that $F=\mu_\Pcal$ because they are monic polynomials, hence the localization of $\Z[\theta]$ at $\Pcal$ is not a discrete valuation ring by Theorem \ref{Hilfsaatziffintro} (1 implies 4).

\section{Application to a tower of nested quadratic square roots}\label{torre}

%
%
%

Recall that the sequence $(x_n)_{n\ge0}$ is defined by induction by $x_0=0$ and $x_{n+1}=\sqrt{\nu+x_n}$, where $\nu\ge3$ is an integer. Given a maximal ideal $\Pcal$ in $\Z[x_n]$, we will write $\mu_n=\mu_\Pcal$.

From now on we assume that $\nu$ is square-free and congruent to $2$ or $3$ modulo $4$ (we will use various facts from \cite{VyV} and \cite{Cas18} which make this assumption, which is necessary and sufficient for the the tower $\cup \Q(x_n)$ to increase at each step and for $\Z[x_1]$ to be the ring of integers of $\Q(x_1)$). 

We will use the following result from \cite[proof of Proposition 2.15]{VyV} and \cite[Proposition 3.2.1]{Cas18}.

\begin{proposition}\label{2-eis} For each $n\geq 1$, let $P_n$ be the minimal polynomial of $x_n$. We have:
\begin{enumerate}
	\item if $n$ is odd, then $P_n(t+a)$ is $2$-Eisenstein, where 
	$$
	a=
	\begin{cases}
	0 &\text{ if } \nu\equiv 2 \mod 4\\
	1 &\text{ if } \nu\equiv 3 \mod 4.
	\end{cases}
	$$ 
	\item If $n$ is even, then $P_n(t)$ is $2$-Eisenstein.
	\item $\disc(x_0)=1$, $\disc(x_1)=4\nu$ and $\disc(x_{n})=(\disc(x_{n-1}))^2 2^{2^n}P_n(0)$ for every $n\ge2$.
\end{enumerate}
Moreover, writing $f(t)=t^2-\nu$, we have $P_n(t)=f^{\circ n}(t)$, hence in particular $P_n$ has no monomial of odd degree. 
\end{proposition}

We will also need the following corollary. 

\begin{corollary}\label{nu-eis}
Let $p$ be a prime divisor of $\nu$. For each $n\geq 1$, $P_n$ is $p$-Eisenstein.
\end{corollary}
\begin{proof}
From the fact that $P_n(t)=P_1^{\circ n}(t)$, with $P_1(t)=t^2-\nu$, we have 
$$
P_{n+1}(t)=P_n\circ P_1(t)=t^{2^n}+\sum_{\ell=1}^{2^n-1}a_{\ell}P_1(t)^{\ell}+a_0
$$ 
for some integers $a_\ell$. Assume that $P_n$ is $p$-Eisenstein, so that $p$ divides each $a_\ell$ and $p^2$ does not divide $a_0$. The constant term of $P_{n+1}(t)$ being
$$
\sum_{\ell=1}^{2^n-1}a_{\ell}P_1(0)^{\ell}+a_0,
$$
if $p^2$ would divide it, then it would divide $a_0=P_n(0)$ (because $\nu$ divides $P_1(0)$ and each $a_\ell$), which would contradict the hypothesis of induction. 
\end{proof}


%


\subsection{Proof of item 1 of Theorem \ref{p2intro}}


\subsubsection{The cases $p=2$ and $p$ a prime divisor of $\nu$}

In this subsection, we will use $i_1$, $i_2$, $\pi_1$ and $\pi_2$ as in Section \ref{secHilfsaatz}:
$$
\begin{tikzcd}
\Z[x_n] \arrow[r, twoheadrightarrow, "\pi_2"] 
& \Z[x_n]/\Pcal \\
\Z \arrow[u, hook, "i_1"] \arrow[r, twoheadrightarrow, "\pi_1"]
&\F_p \arrow[u, hook, "i_2"']
\end{tikzcd}
$$

\begin{lemma} \label{pdnu}
Let $n\ge 1$ be an integer. Let $p$ be a prime divisor of $\nu$. For any maximal ideal $\Pcal\subseteq\Z[x_n]$ lying over $p\Z$, the localization of $\Z[x_n]$ at $\Pcal$ is a discrete valuation ring. 
\end{lemma}
\begin{proof}
Since $P_n$ is $p$-Eisenstein (by Corollary \ref{nu-eis}), we may write 
$$
P_n(t)=t h(t)+p m,
$$
with $m$ relatively prime to $p$ (note that $m\ne0$). We show that $x_n$ lies in $\Pcal$ (so $\mu_n(t)=t$, and we can conclude by Theorem \ref{Hilfsaatzif} because $\bar m$ is a non-zero constant). Indeed we have $x_1^2=\nu\in p\Z\subseteq\Pcal$, hence $x_1\in\Pcal$ because $\Pcal$ is prime. Since $x_n^2=\nu+x_{n-1}\in\Pcal$, it is immediate by induction that $x_n$ lies in $\Pcal$. 
\end{proof}

\begin{lemma} \label{22222}
Let $n\ge 1$ be an integer. For any $\Pcal$ maximal over $2\Z$, the localization of $\Z[x_n]$ at $\Pcal$ is a discrete valuation ring. 
\end{lemma}
\begin{proof}
Recall that $\nu$ is congruent to $2$ or $3$ modulo $4$. We have 
$$
\pi_2(x_1)^2=\pi_2(\nu)
=i_2(\pi_1(\nu))=i_2(\nu+2\Z)
=\begin{cases}
\Pcal&\textrm{if $\nu$ is even,}\\
1+\Pcal&\textrm{if $\nu$ is odd.}
\end{cases}
$$
If $\nu$ is odd, then $\pi_2(x_1)=\pm 1+\Pcal=1+\Pcal$. For any integer $k\ge1$, we have 
$$
\pi_2(x_{k+1})^2=\pi_2(\nu+x_k)=
\begin{cases}
\pi_2(x_k)&\textrm{if $\nu$ is even}\\
1+\pi_2(x_k)&\textrm{if $\nu$ is odd}. 
\end{cases}
$$ 
In particular, for even $\nu$ we have $\pi_2(x_n)=\Pcal$ for each $n$. If $\nu$ is odd, then we have $\pi_2(x_2)^2=1+1+\Pcal=\Pcal$, hence $\pi_2(x_2)=\Pcal$. Therefore, for odd $\nu$, we have $\pi_2(x_n)=1+\Pcal$ for odd $n$, and $\pi_2(x_n)=\Pcal$ for even $n$. 

For $n$ even, and for any $n$ in the case that $\nu$ is even, this proves that $\mu_n(t)=t$. Since the polynomial $P_n(t)$ is $2$-Eisenstein (Proposition \ref{2-eis}), we can write 
$$
P_n(t)=t h(t)+2m
$$
with $m$ odd. We have then $\gcd(\bar\mu_n,\bar m)=1$, so we are done by Theorem \ref{Hilfsaatzif}. 

For $n\ge1$ odd and $\nu$ odd, since $\pi_2(x_n+1)=\Pcal$, we can take $\mu_n(t)=t-1$. 

By Proposition \ref{2-eis}, $P_n(t+1)$ is $2$-Eisenstein for $n$ odd and $\nu$ congruent to $3$ modulo $4$, so we can write
$$
P_n(t+1)=t h(t)+2m
$$
with $m$ odd, hence
$$
P_n(y)=\mu_n(y)h(y-1)+2m,
$$
and we have $\gcd(\bar\mu_n,\bar m)=1$, so the lemma is proven for odd $\nu$ and $n$, again thanks to Theorem \ref{Hilfsaatzif}.

\end{proof}

\subsubsection{General prime $p$}

Fix an odd prime $p$ which does not divide $\nu$. Recall that $C_n$ denotes the constant term of $P_n$. If $p$ does not divide any $C_n$, then it does not divide the discriminant of any $x_n$ by item 3 of Proposition \ref{2-eis}, hence $\Z[x_n]_{\Pcal}$ is a discrete valuation ring for any $n$ by Remark \ref{UchLunRem}. So we assume that $p$ divides at least one $C_n$ and we let $m=n(p)$ be the least positive integer such that $p$ divides $C_m$. Note that we can assume $p\ne2$ and $n(p)\ge2$ thanks to Lemmas \ref{pdnu} and \ref{22222}. 

By Proposition \ref{2-eis}, we know that $P_n$ has no monomial of odd degree. We have 
$$
P_0(t)=t,\quad P_1(t)=t^2-\nu
\quad\textrm{and}\quad P_2(t)=t^4-2\nu t^2+\nu^2-\nu.
$$
For $n\ge 1$, write $D_n$ for the coefficient of $t^2$ in $P_n$. Let $R_n\in\Z[t]$ be such that 
$$
P_n(t)=R_n(t)t^4+D_nt^2+C_n
$$ 
For any integer $k$, let $s(k)$ denote the set of primes that divide $k$. 

\begin{lemma}\label{lemgenpx2}
For every $n\ge 2$, we have $\bigcup_{1\le k\le n} s(C_k)= s(D_{n+1})$. 
\end{lemma}
\begin{proof}
We have 
$$
R_1(t)=0, \quad D_1=1\quad C_1=-\nu,
$$
$$
R_2(t)=1, \quad D_2=-2\nu\quad C_2=\nu^2-\nu.
$$
Since $x_0=0$, by Proposition \ref{2-eis}, for every $n\ge 1$, we have
$$
\begin{aligned}
P_{n+1}&= P_1\circ P_n=P_n^2+C_1\\
       &=(R_n(t)t^4+D_nt^2+C_n)^2+C_1\\
       &=R_{n+1}(t)t^4+2C_nD_nt^2+C_n^2+C_1, 
\end{aligned}
$$
hence $D_{n+1}=2C_nD_n$. For $n=2$, this gives $D_3=2C_2D_2=-4\nu(\nu^2-\nu)$, hence $s(D_3)=s(C_1)\cup s(C_2)$ (because $\nu^2-\nu$ is even, so $2$ divides $C_2$), so the formula is true for $n=2$. Since $D_{n+1}=2C_nD_n$ and $2\in s(D_n)$, we can conclude immediately by induction on $n\ge 2$. 
\end{proof}

For any $n\ge 0$ and $k\ge 1$ integers, we have 
$$
P_{n+k}=P_k\circ P_n=P_n^2(P_n^2R_k(P_n)+D_k)+C_k, 
$$
hence, with $k=n(p)$, we obtain
$$
\bar P_{n+n(p)}=\bar P_{n(p)}\circ \bar P_n=\bar P_n^2(\bar P_n^2\bar R_{n(p)}(\bar P_n)+\bar D_{n(p)})
$$
in $\F_p[t]$. We deduce by induction that we have divisibility sequences for every $r\ge0$:
$$
\bar P_r | \bar P_{r+n(p)} | \bar P_{r+2n(p)}|\cdots.
$$
Also for any $n\ge 1$ we have
$$
\bar P_{n+n(p)}=\bar P_n\circ \bar P_{n(p)}=\bar P_{n(p)}^2(\bar P_{n(p)}^2\bar R_n(\bar P_{n(p)})+\bar D_n)+\bar C_n,
$$
hence $\bar C_{n+{n(p)}}=\bar C_n$, so we deduce by induction that if $r$ is the remainder of the division of $n$ by $n(p)$, then $\bar C_n=\bar C_{r}$ (note that $r$ may be $0$).

\begin{lemma}\label{lempprel}
For any $k$ and $\ell$ non-negative integers which are distinct modulo $n(p)$, we have $\bar P_k$ coprime with $\bar P_\ell$. 
\end{lemma}
\begin{proof}
Write $\ell=qn(p)+r$ and $k=q'n(p)+r'$, with $r$ and $r'$ such that $0\le r,r'<n(p)$. Assume $\ell>k$ (so $\ell-k\ge1$). We have 
$$
\bar P_{\ell}=\bar P_{\ell-k}\circ \bar P_{k}=\bar P_{k}^2(\bar P_{k}^2\bar R_{\ell-k}(\bar P_k)+\bar D_{\ell-k})+\bar C_{\ell-k},
$$
so if $\rho$ divides $\bar P_k$ and $\bar P_\ell$, then it divides $\bar C_{\ell-k}=\bar C_{|r-r'|}$, which is not $\bar 0$ (because $n(p)$ is the least positive integer such that $C_{n(p)}$ is divisible by $p$).
\end{proof}

By Lemma \ref{lemgenpx2} and the definition of $n(p)$, $p$ does not divide $D_{k}$ for $0<k\le n(p)$, so $\bar D_k\ne\bar 0$ for those $k$, and we deduce that $\bar P_n$ is coprime with $\bar P_n^2\bar R_k(\bar P_n)+\bar D_k$ for any $n\ge0$. 

\begin{lemma}\label{lempsep}
For any $n\ge0$, the polynomial $\bar P_n^2\bar R_{n(p)}(\bar P_n)+\bar D_{n(p)}$ is separable. 
\end{lemma}
\begin{proof}
Let $\rho\in\F_p[t]$ be irreducible. Recall that we have
$$
\bar P_{n+n(p)}=\bar P_n^2(\bar P_n^2\bar R_{n(p)}(\bar P_n)+\bar D_{n(p)}).
$$
For the sake of contradiction, assume that $\rho^2$ divides $\bar P_n^2\bar R_{n(p)}(\bar P_n)+\bar D_{n(p)}$ (so in particular $\rho$ does not divide $\bar P_n$). Since $\rho^2$ divides $\bar P_{n+n(p)}$, $\rho$ divides its derivative $\bar P_{n+n(p)}'$. Since $\bar P_{n+n(p)}=\bar P_{n+n(p)-1}^2-\nu$, we have $\bar P_{n+n(p)}'=2\bar P_{n+n(p)-1}' \bar P_{n+n(p)-1}$, and one easily proves by induction that we have 
$$
\bar P_{n+n(p)}'=2^{n+n(p)} \bar P_{0}' \prod_{\ell=1}^{n+n(p)} \bar P_{n+n(p)-\ell}.
$$
Since $\rho$ divides $\bar P_{n+n(p)}$, by Lemma \ref{lempprel}, it does not divide $\bar P_{n+n(p)-\ell}$ for any $\ell$ which is not a multiple of $n(p)$, and it does not divide $\bar P_n$, hence it does not divide any $\bar P_k$ with $k\le n$ congruent to $n$ modulo $n(p)$. 
Therefore, $\rho$ divides $\bar P_0'(t)=1$, which is absurd. 
\end{proof}

We can now conclude the proof of item 1 of Theorem \ref{p2intro}. 

\begin{proposition}\label{p2}
Assume $x_0=0$ and let $n\ge1$ be an integer. If $p^2$ does not divide $C_{n(p)}$, then for any $\Pcal$ maximal over $p\Z$, the localization of $\Z[x_n]$ at $\Pcal$ is a discrete valuation ring. 
\end{proposition}
\begin{proof}
For any $\mu_n\in\Z[t]$ of minimal degree and monic such that $\mu_n(x_n)\in\Pcal$, we have already proven that $\bar\mu_n\in\F_p[t]$ is the minimal polynomial of $\pi_2(x_n)$. 

Assume $n\ge n(p)$, so we can write
$$
\bar P_{n}=\bar P_{n(p)}\circ\bar P_{n-n(p)}=\bar P_{n-n(p)}^2(\bar P_{n-n(p)}^2\bar R_{n(p)}(\bar P_{n-n(p)})+\bar D_{n(p)}).
$$ 
Note that if $n=n(p)$, then $P_{n-n(p)}$ is just $t$, and we let the reader convince her/himself that the following argument works fine also in that case. If $\bar\mu_n$ divides 
$$
\bar P_{n-n(p)}^2\bar R_{n(p)}(\bar P_{n-n(p)})+\bar D_{n(p)},
$$ 
then on the one hand it does not divide $\bar P_{n-n(p)}^2$ (because $p$ does not divide $D_{n(p)}$), and on the other hand, by Lemma \ref{lempsep}, it appears with multiplicity $1$ in the factorization of $\bar P_{n-n(p)}^2\bar R_{n(p)}(\bar P_{n-n(p)})+\bar D_{n(p)}$, so it lies in the square-free part of $\bar P_n$ and we are done by Remark \ref{UchLunRem}. So assume that $\bar\mu_n$ does not divide $\bar P_{n-n(p)}^2\bar R_{n(p)}(\bar P_{n-n(p)})+\bar D_{n(p)}$, so it divides $\bar P_{n-n(p)}$, and write $P_{n-n(p)}=\mu_nh+pg_0$. We have
$$
\begin{aligned}
P_n&=P_{n-n(p)}^2[P_{n-n(p)}^2R_{n(p)}(P_{n-n(p)})+D_{n(p)}]+C_{n(p)}\\
&=(\mu_nh+pg_0)^2[(\mu_nh+pg_0)^2R_{n(p)}(P_{n-n(p)})+D_{n(p)}]+C_{n(p)}\\
&=\mu_n(\dots)+p^2g_0^2(\dots)+C_{n(p)}\\
&=\mu_n(\dots)+p(pg_0^2(\dots)+C_{n(p)}/p)\\
&=\mu_n(\dots)+pg,
\end{aligned}
$$
where $\bar g$ is a constant distinct from $0$ modulo $p$ by hypothesis. We conclude by Theorem \ref{Hilfsaatzif}.

Assume $n<n(p)$. We have \cite[Prop. 3.2.1]{Cas18}
$$
\disc(x_0)=1 \text{ and } \disc(x_1)=2^2\nu,
$$
and for any $k\geq 2$ 
$$
\disc(x_n)=(\disc(x_{n-1}))^2\cdot 2^{2^n} C_n,
$$
and since for each $k\le n$, $p$ does not divide $C_{k}$, it does not divide the discriminant of $x_n$ (because $p\ne 2$), so we are done by Remark \ref{UchLunRem}.
\end{proof}

\subsection{Proof of item 2 of Theorem \ref{p2intro}}

We start with a simple lemma. 

\begin{lemma}\label{p2back0}
For any rational prime $p$, the ideal $(p,x_{n(p)})$ of $\Z[x_{n(p)}]$ is proper. 
\end{lemma}
\begin{proof}
Write $n=n(p)$. If not, then we would have $pb(x_n)+x_na(x_n)=1$ for some $a(t),b(t)\in\Z[t]$. Let $W\in\Z[t]$ be such that $P_nW=pb(t)-ta(t)-1$ (such $W$ exists because $P_n$ is the minimal polynomial of $x_n$). Modulo $p$, the constant term on the right-hand side is $-1$, while the constant term on the left-hand side is $0$ (because $p$ divides $C_{n(p)}$).
\end{proof}

The following proposition proves item 2 of Theorem \ref{p2intro}. 

\begin{proposition}\label{p2back}
Let $p$ be a rational prime. If every maximal ideal $\Pcal$ of $\Z[x_{n(p)}]$ above $p$ the localization of $\Z[x_{n(p)}]$ at $\Pcal$ is a discrete valuation ring, then $p^2$ does not divide $C_{n(p)}$. 
\end{proposition}
\begin{proof}
Write $n=n(p)$. Since $(p,x_n)$ is a proper ideal of $\Z[x_n]$, it is contained in a maximal ideal $\Pcal$, which lies above $p$. Since $x_n$ lies in $\Pcal$, we have $\bar\mu_n=t$. By Corollary \ref{Hilfsaatziff}, there exists $g,h\in\Z[t]$ such that $P_n=\mu_nh+pg$ and $\gcd(t,\bar h,\bar g)$ is $1$, which is equivalent to the fact that either $\gcd(t,\bar h)$ or $\gcd(t,\bar g)$ is $1$. 

The relation 
\begin{equation}\label{eqback}
\mu_nh+pg=P_n=t^2(\dots)+C_n
\end{equation} 
gives $t\bar h=t^2(\dots)$ modulo $p$, so $t$ divides $\bar h$. Therefore, we have $\gcd(t,\bar g)=1$, and in particular the constant term of $g$ is not divisible by $p$. Evaluating in $0$ in Equation \eqref{eqback}, we obtain $pg(0)=C_n$, hence $p^2$ does not divide $C_n$. 
\end{proof}

%
%

\noindent Xavier Vidaux (corresponding author)\\
Universidad de Concepci\'on, Concepci\'on, Chile\\
Facultad de Ciencias F\'isicas y Matem\'aticas\\
Departamento de Matem\'atica\\
Casilla 160 C\\
Email: xvidaux@udec.cl\\

\noindent Carlos R. Videla\\
Mount Royal University, Calgary, Canada\\
Department of Mathematics and Computing\\
email: cvidela@mtroyal.ca

\end{document}